\documentclass[12pt, twoside]{article}

\usepackage{amsmath,amsthm,amssymb}

\usepackage{times}

\usepackage{enumerate}

 \usepackage{amsmath}
\usepackage{amssymb}
\usepackage{amsthm}
\usepackage{graphicx}
\usepackage{amscd}
\usepackage{graphicx}

\usepackage{amssymb}
\usepackage{amsmath}
\usepackage{latexsym}
\usepackage[numbers,sort&compress]{natbib}
\usepackage{verbatim}
\usepackage{physics}
\usepackage{dsfont}	
\usepackage{microtype}  
\usepackage{hyperref}
\usepackage{tikz}
\usepackage{booktabs}   

\usepackage{color}

\date{}

\newtheorem{theorem}{Theorem}

\title{A polynomial time reduction from the multi-graph isomorphism problem to additive code equivalence}

\begin{document}
\author{Simeon Ball\thanks{30 August 2021. The first author acknowledges the support of the Spanish Ministry of Science and Innovation grant MTM2017-82166-P.} \ and James Dixon}
\maketitle

\begin{abstract}
We present a polynomial time reduction from the multi-graph isomorphism problem to the problem of code equivalence of additive codes over finite extensions of ${\mathbb F}_2$.
\end{abstract}

\section{Introduction}
Let $F$ be a finite field of characteristic $p$, where $p$ is a prime. An {\em additive code} of length $n$ over $F$ is a subset of $F^n$ with the property that for all $u,v \in C$, we have that $u+v \in C$. It is easy to prove that an additive code over $F$ is linear over ${\mathbb F}_p$, the finite field with $p$ elements. Thus, an additive code over $F$ can be defined as the row space over ${\mathbb F}_p$ of a matrix whose elements are from $F$. The code equivalence problem for additive codes is the following. Given two $k \times n$ matrices over $F$, when is there a permutation of the columns of one of the matrices, together with a permutation $\sigma_i$, for $i \in \{1,\ldots,n\}$, which is a permutation of the elements of $F$ in the $i$-th coordinate, so that the row-spaces over ${\mathbb F}_p$ are the same. If $F={\mathbb F}_p$ then an additive code is linear and if, furthermore, $p=2$ then the code is a binary linear code. If the permitations $\sigma_i$ are all additive then the codes are {\em additively equivalent}. It is not known if equivalent additive codes are necessarily additively equivalent, see \cite{BD2021}.

A multi-graph is a graph whose edges are assigned a weight from $\{1,\ldots,h\}$, for some natural number $h$. An edge joining vertices $u$ and $v$ of weight $w$ is interpreted as $w$ edges joining $u$ and $v$. The multi-graph isomorphism problem is the following. Given two multi-graphs determine if there a bijection between the set of vertices which induces a bijection on the edges. i.e. edges of weight $w$ are mapped to edges of weight $w$. If $h=1$ then a  multi-graph is simply a graph.

In \cite{PR1997}, Petrank and Roth provide a polynomial time reduction from the graph isomorphism problem to the binary linear code equivalence problem. The graph isomorphism problem is in NP but is not known to be NP-complete. The problem is not known to be solvable in polynomial time either and is therefore a good candidate to belong to the computational complexity class NP-intermediate. 
Known algorithms for graph isomorphism include McKay's Nauty algorithm \cite{McKay1981}, Ullmann's algorithm \cite{Ullmann1976}, the VF2 algorithm \cite{CFSV2004} and the parameterised matching algorithm \cite{MKEHHP2013}. All these algorithms have exponential worst case performance. Solving isomorphism generally takes much longer time if there is no match, since all possible mappings are eventually searched until it is shown that there is no isomorphism. The latter extends to multi-graph isomorphism and is based on a parameterised sequence which is a walk that covers every vertex in the graph.

In this note we extend the polynomial time reduction of Petrank and Roth \cite{PR1997} to a polynomial time reduction of multi-graph isomorphism to additive code equivalence, where the code is over an extension of ${\mathbb F}_{2}$. Since a graph on $n$ vertices has at most $\frac{1}{2}n(n-1)$ edges, we assume that the multi-graph has $h \leqslant \frac{1}{2}n(n-1)$ weights and that these are from the set $\{1,\ldots, h\}$.

\section{The reduction of multi-graph isomorphism to additive code equivalence}

Let $\Gamma$ be a multi-graph with vertex set $V$, whose edges have weights belonging to the set $\{1,\ldots,h\}$. Let $A$ be the incidence matrix of $\Gamma$, whose rows are indexed by the edges, whose columns are indexed by the vertices and where the edge-vertex entry is equal to the weight of the edge if the edge is incident with the vertex and zero otherwise. Let $E$ denote the set of edges of $\Gamma$ and let $E_i$ be the subset of $E$ of edges of weight at least $i$ (so $E_1=E$).

Let $e$ be a primitive element of ${\mathbb F}_{2^r}$, where $r>\log_2 h$. Let $\mathrm{D}_i$ be the $|E_i| \times |E_i|$ diagonal matrix indexed by the edges of $E_i$ whose diagonal entry is $e^{j-1}$, where $j$ is the weight of the edge indexing the row of $\mathrm{D}_i$.

Let
$$
N=(h+2)|E_1|+|E_2|+\cdots+|E_{h}|+|V|.
$$
We map $\Gamma$ to the additive code which is the ${\mathbb F}_2$-row span of the $|E| \times N$ matrix
$$
\mathrm{G}(\Gamma)=\left( \underbrace{D_1 \ | \ D_1 \ | \ \cdots \ | \ D_1}_{h+2} \ | \ \begin{array}{c} O_2 \\ D_2 \end{array}
\ | \ \begin{array}{c} O_3 \\ D_3  \end{array} \ | \ \cdots \ | \ \begin{array}{c} O_{h} \\ D_{h}  \end{array}
\ | \ A
\right)
$$
where $O_i$ is a matrix of zeros, whose dimensions are determined by the fact that the matrix $\mathrm{G}(\Gamma)$ is a $|E| \times N$ matrix and the matrix $\mathrm{D}_i$ is a $|E_i| \times |E_i|$ matrix.

The following two observations will be vital.
\begin{enumerate}
\item[(O1)]
The row of $G(\Gamma)$ indexed by an edge of weight $i$ is a codeword of weight $h+3+i$ whose coordinates are either $0$ or $e^{i-1}$.
\item[(O2)]
The only codewords of weight at most $2h+3$ are rows of $G(\Gamma)$.
\end{enumerate}

The following is the main theorem of this note.

\begin{theorem} \label{mainthm}
The multi-graphs $\Gamma$ and $\Gamma'$ are isomorphic if and only if the additive codes generated by $\mathrm{G}(\Gamma)$ and $\mathrm{G}(\Gamma')$ are equivalent.
\end{theorem}

\begin{proof}

Suppose the multi-graphs $\Gamma$ and $\Gamma'$ are isomorphic. Then there is a permutation of the columns of $\mathrm{A}$ and a permutation of the rows of $\mathrm{A}$, which gives $A'$, the incidence matrix of the multi-graph $\Gamma'$. Apply the column permutation to the last $|V|$ columns of $\Gamma$ and the permutation of the rows of $\mathrm{A}$ to the rows of $\mathrm{G}(\Gamma)$. Since the permutation of the rows of $\mathrm{A}$, takes edges of weight $i$ to edges of weight $i$, it takes codewords of weight $h+3+i$, which are rows of $G(\Gamma)$, to codewords of weight $h+3+i$, which are rows of $G(\Gamma')$ by (O1) and (O2). Thus, $G(\Gamma)$ and $G(\Gamma')$ generate equivalent codes.

Now suppose the additive codes generated by $\mathrm{G}(\Gamma)$ and $\mathrm{G}(\Gamma')$ are equivalent. Then, by definition, there is an $N \times N$ permutation matrix $P$, a $|E| \times |E|$ change of basis matrix $S$ and permutations $\sigma_i$, $i\in \{1,\ldots,N\}$, such that we can apply $P$ to the columns of $\mathrm{G}(\Gamma)$, $S$ to the rows of $\mathrm{G}(\Gamma)$ and the permutation $\sigma_i$ to the $i$-th coordinate and obtain a matrix whose generates, over ${\mathbb F}_2$, the same code as $\mathrm{G}(\Gamma')$.

Properties (O1) and (O2) imply that the change of basis matrix $S$ is in fact a permutation matrix, since rows of $\mathrm{G}(\Gamma)$ must be mapped to rows of $\mathrm{G}(\Gamma')$. 
Furthermore, $S$ maps codewords of weight $h+3+i$ to codewords of weight $h+3+i$, so the initial $N-|V|$ columns of $SG$ contain each vector of weight one, with a non-zero coordinate indexed by an edge of weight $i$, precisely $h+1+i$ times. Hence, the first $N-|V|$ columns of $G'$ can be obtained by permuting the first $N-|V|$ columns of $SG$. This permutation of the columns then maps the columns of $A$ to columns of $A'$. 

The permutation $\sigma_i$ must fix all elements of $F$ appearing in the $i$-th column of a column of $A$, since by property (O1) the rows of $\mathrm{G}(\Gamma')$ of weight $h+3+i$ have coordinates $0$ or $e^{i-1}$.

Hence, we obtain $\mathrm{G}(\Gamma')$ from $\mathrm{G}(\Gamma)$ by applying a permutation of the columns of $A$ and a permutation of the rows of $A$. Thus, $\Gamma$ and $\Gamma'$ are isomorphic.
 \end{proof}

The polynomial time reduction follows from Theorem~\ref{mainthm}, by noting that $|N|<chn^2<n^4$, for some constant $c$. 
Note also that we can replace equivalence with additive equivalence, since the permutations $\sigma_i$ are always trivial.

It is interesting to ask if $N$ could be decreased, possibly by increasing the size of the field extension. For $h=2$ it is possible to reduce $N$ to $3|E_1|+|E_2|+|V|$ by simply deleting one of the initial $D_1$ matrices. One can check that the proof still works, although a little more subtlety is required. Observe that the sum of two codewords of weight $5$ (corresponding to edges of weight one) cannot produce a codeword of weight $6$ since this would require a repeated single edge, which is an edge of weight two.

\section{Code equivalence algorithms}

We discussed known algorithms for solving the graph isomorphism problem in the introduction. Algorithms for solving binary linear code equivalence include Bouyukliev's algorithm \cite{Bouyukliev2007}, which is similar to McKay's graph isomorphism algorithm \cite{McKay1981}, Leon's algorithm \cite{Leon1982} and the support splitting algorithm of Sendrier \cite{Sendrier1999}, \cite{Sendrier2000}. The latter does not apply to all binary linear codes but it is interesting because it is fast for binary linear codes in which the dimension of the hull $H(C)=C \cap C^{\perp}$ is small. It seems that this is the case for codes obtained from the polynomial time reduction of the graph isomorphism problem due to Petrank and Roth \cite{PR1997}. This warrants further investigation.

The support splitting algorithm allows one to determine signatures from the weight distribution of the hull of truncated binary linear codes, truncating one coordinate at a time.  This works because one of the hulls of either the truncated code or the truncations of its dual is non-trivial. This can be seen as follows. Let $C_i$ be the code obtained from $C$ be setting the $i$-th coordinate to zero in all codewords. Suppose that $H(C)=\{0\}$. Thus, ${\mathbb F}_2^n=C\oplus C^{\perp}$.

Without loss of generality, we consider a truncation on the first coordinate. We have that 
$$
(1,0,\ldots,0)=u+v,
$$
for some $u \in C$ and $v \in C^{\perp}$. 

If $(u_1,v_1)=(1,0)$ then, since $v_1=0$, we have that $v \in (C_1)^{\perp}$ and since $v= (1,0,\ldots,0)+u$, we have that $v \in C_1$. Hence, $v \in H(C_1)$.

Similarly, if $(u_1,v_1)=(0,1)$ then, since $u_1=0$, we have that $u \in ((C^{\perp})_1)^{\perp}$ and since $u= (1,0,\ldots,0)+v$, we have that $u \in (C^{\perp})_1$. Hence, $u \in H((C^{\perp})_1)$.

This truncation trick does not carry over to additive codes. It is possible that both $H((C^{\perp})_1)$ and $H(C_1)$ are trivial. However, if we are interested in establishing additive equivalence then we can employ a slightly modified support splitting algorithm; instead of removing the $i$-th coordinate we take a subspace. In effect, truncating in the binary case would be equivalent to taking the subspace $\{0\}$. This is done in the following way. The $i$-th column of a generator matrix for $C$ is a vector $v_i \in {\mathbb F}_{2^r}^k$, where $|C|=2^k$. Writing $v_i$ out over the basis $\{1,\alpha,\ldots,\alpha^{r-1}\}$ of ${\mathbb F}_{2^r}$ over ${\mathbb F}_2$,
$$
v_i=\sum_{j=0}^{r-1} v_i^{(j)}\alpha^j,
$$
where $v_i^{(j)} \in {\mathbb F}_{2}^k$. Let $\rho_i$ be the subspace of  ${\mathbb F}_{2}^k$, of dimension at most $r$, spanned by $v_i^{(0)},\ldots, v_i^{(r-1)}$. The subspace $\rho_i$ is unaffected by an additive permutation on the $i$-th coordinate, which only has the effect of changing the basis. Thus, we can select a subspace $\pi_i$ of $\rho_i$ and replace the $i$-th column of the generator matrix with the corresponding vector of ${\mathbb F}_{2^r}^k$ obtained from the subspace $\pi_i$ by reversing the construction of $\rho_i$ above. This code we denote by $C_{i,\pi_i}$, where $i$ is the selected coordinate and $\pi_i$ is the subspace. For equivalent codes, we have that selecting a coordinate and a subspace in $C$, there must be a corresponding coordinate and subspace in $C'$ which is equivalent, i.e. $C_{i,\pi_i}$ is equivalent to $C'_{i',\pi'_{i'}}$, for some $i'$ and $\pi'_{i'}$. The weight distribution of the hull of $C_{i,\pi_i}$ can then be used as a signature with an aim of establishing the permutation taking $C$ to $C'$, as in the support splitting algorithm.

\vspace{1cm}

   Simeon Ball\\
   Departament de Matem\`atiques, \\
Universitat Polit\`ecnica de Catalunya, \\
Carrer Jordi Girona 1-3,\\
08034 Barcelona, Spain \\
   {\tt simeon.michael.ball@upc.edu} \\

James Dixon\\
Facultat de Matem\`atiques, \\
Universitat Polit\`ecnica de Catalunya, \\
Carrer de Pau Gargallo, 14, \\
08028 Barcelona, Spain\\
  {\tt james.dixon@estudiantat.upc.edu} \\

\end{document}